\documentclass[amssymb,amsfonts,refcheck,12pt,verbatim,righttag]{amsart}

\usepackage{placeins}
\usepackage{float}
\usepackage{booktabs}
\usepackage{pgfplots}
\usepackage{pgfplotstable}
\pgfplotsset{compat=1.17}
\usepackage{tikz,pgfplots}
\usepackage{xcolor}
\usepackage{subcaption}
\usetikzlibrary{patterns,decorations.pathreplacing}
\usetikzlibrary{arrows.meta, decorations.markings, calc, decorations.pathreplacing}
\usepackage{graphicx}
\usepackage{graphics,color}
\usepackage{amsmath}
\usepackage{amsfonts}
\usepackage{color}
\usepackage{amssymb}
\usepackage{cancel}

\numberwithin{equation}{section}  
\theoremstyle{plain}

\newtheorem{thm}{Theorem}[section]
\newtheorem{lem}[thm]{Lemma}
\newtheorem{pro}[thm]{Proposition}

\newtheorem{ex}[thm]{Example}

\def\R {{\Bbb R}}
\def\N {{\Bbb N}}

\def\A {{\mathcal A}}

\newcommand{\set}[1]{\left\{#1\right\}}



\usepackage[colorlinks,linkcolor=blue,citecolor=blue,pagebackref,urlcolor=black,hypertexnames=false]{hyperref}

\usepackage{float}
\begin{document}
	\baselineskip 15pt
	\baselineskip 13.7pt

\newcommand{\Ptr}{\mathbb P}
\newcommand{\Etr}{\mathbb E}
\newcommand{\Otr}{\Omega}

\title{Random Survivor Sets with Moving Holes in Bedford--McMullen Carpets}

\author[D.-R.  Kong]{Derong Kong}
\address[D.-R.  Kong]{College of Mathematics and Statistics, Center of Mathematics, Chongqing University, Chongqing, 401331, P.R.China}
\email{derongkong@126.com}

\author[C.-Y. Ma]{Cai-Yun Ma}
\address[C.-Y. Ma]{ HUN-REN–BME Stochastics Research Group, Müegyetem rkp. 3., H-1111 Budapest, Hungary}
\email{macaiyun813@gmail.com}

\author[K. Simon]{K\'{a}roly Simon}
\address[K. Simon] { Department of Stochastics, Institute of Mathematics, Budapest University of Technology and
	Economics, Müegyetem rkp. 3., H-1111 Budapest, Hungary \& HUN-REN–BME Stochastics Research Group, Müegyetem rkp. 3., H-1111 Budapest, Hungary
}
\email{simonk@math.bme.hu}



\begin{abstract}
	Let $2\le m < n$ be integers, and let $K(\mathcal{A})$ be the Bedford--McMullen carpet generated by $\{(x,y)\mapsto((x+i)/m,(y+j)/n):(i,j)\in\mathcal{A}\}$, where $\mathcal{A}$ is a subset of $\{0,1,\ldots,m-1\}\times\{0,1,\ldots,n-1\}$. Write $\mathcal{B}$ for the set of occupied columns, and assume that $\#\mathcal{B}\ge2$ and each column contains at least two symbols of $\mathcal{A}$. At each vertex of the column tree $\mathcal{B}^*$, we independently choose one forbidden symbol from $\mathcal{A}$ according to a prescribed distribution. Let $\mathbb{P}$ denote the resulting product measure on $\mathcal{A}^{\mathcal{B}^*}$. This gives a random survivor set $K_\omega(\mathcal{A})$ with $\omega\in\mathcal{A}^{\mathcal{B}^*}$. We show that, for $\mathbb{P}$-a.e. $\omega$,
	\begin{equation*}
		\dim_H K_\omega(\mathcal{A})=\frac{1}{\log m}\log\left(\int \sum_{b\in\mathcal{B}}\ell_b(\omega)^{\frac{\log m}{\log n}} d\mathbb{P}(\omega)\right),
	\end{equation*}
	where $\ell_b(\omega)$ is the number of surviving symbols in column $b$ at the root.
\end{abstract}
 \maketitle

\section{Introduction and main result}\label{sec-1}
In this paper, we study the Hausdorff dimension of survivor sets arising from a randomly moving hole in Bedford--McMullen carpets.

We first recall the deterministic construction. Let $m,n\in\mathbb N$
with $2\le m<n$, and let $\mathcal A$ be a nonempty subset of
$\{0,\ldots,m-1\}\times\{0,\ldots,n-1\}$. 
Consider the affine iterated function system (IFS) 
\begin{equation}\label{def-ifs}
	 \Phi=\left\{\phi_{i,j}(x,y)
	=\left(\frac{x+i}{m},\frac{y+j}{n}\right):(i,j)\in\mathcal A\right\}.
\end{equation}
There is a unique nonempty compact set $K(\mathcal A)\subseteq[0,1]^2$
satisfying
\begin{equation*}
 K(\mathcal A)=\bigcup_{(i,j)\in\mathcal A}\phi_{i,j}(K(\mathcal A)).
\end{equation*}
It is called the
\emph{Bedford--McMullen carpet} generated by $\mathcal A$.
These carpets were introduced independently by  Bedford \cite{Bedford1984} and McMullen
\cite{McMullen1984} independently.

Let $\tau$ be the column projection, defined by
$$
 \tau:\mathcal A\to\{0,\ldots,m-1\},\qquad \tau(i,j)=i.
$$
We extend $\tau$ coordinatewise to finite and infinite words.
Set $\mathcal B:=\tau(\mathcal A)$, {and for $b\in\mathcal B$ we write
$$
 \ell_b:=\#\{j:(b,j)\in\mathcal A\},
$$}
where $\#A$ denotes the cardinality of a finite set $A$.  McMullen \cite{McMullen1984} computed the box and
	Hausdorff dimensions of $K(\mathcal A)$. 
The packing dimension equals the box dimension by
\cite[Corollary~3.9]{Falconer2003}, since every nonempty relatively
open subset of $K(\mathcal A)$ contains a bi-Lipschitz affine image
of $K(\mathcal A)$.   Thus  
$$  \dim_B K(\mathcal A)=\dim_P K(\mathcal A)
=\log_m\#\mathcal B+\log_n\frac{\#\mathcal A}{\#\mathcal B} $$
and
$$   \dim_H K(\mathcal A)
=\frac{1}{\log m}\log\left(\sum_{b\in\mathcal B}\ell_b^\lambda\right),$$
{where $\lambda:=\frac{\log m}{\log n}$.}
In particular, the Hausdorff dimension depends on the distribution of
rectangles among the columns and may differ from the box dimension.
Subsequent work extended the theory to more general self-affine carpets
and affine-invariant sets; see, for example,
\cite{GatzourasLalley1992,Baranski2007,Kenyon1996}.

For random recursive self-similar sets, almost sure dimension formulae
were obtained by Falconer, Graf, and Mauldin--Williams
\cite{Falconer1986,Graf1987,Graf1988,MauldinWilliams1986}.
Gui and Li \cite{Gui2008} studied a random version of
Bedford--McMullen carpets in which the construction pattern is selected
at each level and shared by all rectangles at that level. They determined
the Hausdorff, packing and box dimensions of these sets.

Related questions arise in the study of open dynamical systems. Given a
dynamical system with a hole, the survivor set consists of points whose
forward orbits never enter the hole. This subject goes back to Pianigiani
and Yorke \cite{Pianigiani1979}; see also
\cite{Demers2005,DemersYoung2006,Ferguson2012,Bruin2018}.
For Bedford--McMullen carpets, Ferguson, Jordan and Rams \cite{FergusonJordanRams2015} studied survivor sets associated with fixed Markov holes, as well as related shrinking target problems.
More recently, Kong, Sun and Wang \cite{Kong2026} studied expanding maps with moving cylinder holes and obtained dimension formulae in terms of the growth rates of admissible words.

In our model, each vertex of the column tree is assigned a forbidden
symbol, independently and with a common prescribed distribution.
Sequences with the same preceding column word have the same forbidden
next symbol. Different column words at the same level may therefore give
different restrictions. We determine the almost sure Hausdorff dimension
of the resulting survivor set.

To state our main result, we introduce some notation. Write $\Sigma:=\mathcal A^{\N}$,
$\mathcal A^*:=\bigcup_{k=0}^\infty\mathcal A^k$, and
$\mathcal B^*:=\bigcup_{k=0}^\infty\mathcal B^k$, where
$\mathcal A^0=\mathcal B^0=\{\varepsilon\}$  with $\varepsilon$   the
empty word. The natural coding map
$\Pi:\Sigma\to K(\mathcal A)$ is given by
\begin{equation}\label{eq}
	\Pi((i_k,j_k)_{k\ge1})
	=\left(\sum_{k=1}^\infty\frac{i_k}{m^k},
	\sum_{k=1}^\infty\frac{j_k}{n^k}\right).
\end{equation}

Let $\pi=(\pi_\gamma)_{\gamma\in\mathcal A}$ be a probability vector,
and set $\Omega:=\mathcal A^{\mathcal B^*}$ and
$\mathbb P:=\pi^{\otimes\mathcal B^*}$.  We regard $\mathcal B^*$ as the full rooted $\#\mathcal B$-ary tree,
	with the {empty word $\varepsilon$} as its root; each vertex $v\in\mathcal B^*$
	has exactly $\#\mathcal B$ children, namely $vb$, $b\in\mathcal B$.
	Thus $\Omega$ is the space of all $\mathcal A$-labellings of this tree.   Furthermore, $\mathbb P$ is a probability  distribution on $\Omega$ which is a Bernoulli measure determined by $\pi $.  For each $v\in\mathcal B^*$, define the coordinate random variable  $F_v: \Omega\to \mathcal A$ by 
$$  F_v(\omega)=\omega_v
\qquad \textrm{for}\quad\omega=(\omega_u)_{u\in\mathcal B^*}\in\Omega.  $$
The random variables  $F_v, v\in\mathcal B^*$ are independent and have the common distribution
$\mathbb P(F_v=\gamma)=\pi_\gamma$.
For each $\omega\in\Omega$, let
\begin{equation}\label{eq:symbolic-sigma-omega}
  \Sigma^{\omega}:=
\left\{(d_k)_{k\ge1}\in\Sigma:
d_{k+1}\ne F_{\tau(d_1\cdots d_k)}(\omega)
\text{ for all }k\ge0\right\}, 
\end{equation}
where {for $k=0$   we set $\tau(d_1\cdots d_0)=\tau(\varepsilon)=\varepsilon$}. The corresponding random survivor set is
$K_\omega(\mathcal A):=\Pi(\Sigma^{\omega})$. 
The forbidden symbol depends on the preceding column word, not merely
on its length. Thus the model differs from {that studied by Gui and Li \cite{Gui2008}.}

For $b\in\mathcal B$ and $\omega\in\Omega$, write
\begin{equation}\label{eq-liomega}
 \ell_b(\omega):=\ell_b-\mathbf1_{\{\tau(F_\varepsilon(\omega))=b\}},
\end{equation}
 where $\mathbf 1_{\{\tau(F_\varepsilon(\omega))=b\}}$ is the indicator function, which equals one if $\tau(F_\varepsilon(\omega))=b$, and zero otherwise.  Then $\ell_b(\omega)$ is the number of admissible symbols in column $b$ at the root of $\Sigma^\omega$. 
Throughout the paper we assume that
\begin{equation}\label{eq-assumption}
 \#\mathcal B\ge2,\qquad \textrm{and}\qquad\ell_b\ge2\quad\textrm{for}\quad b\in\mathcal B.
\end{equation}
In particular, every column retains at least one symbol after a deletion,
so $\Sigma^{\omega}$ and $K_\omega(\mathcal A)$ are nonempty compact
sets for every $\omega$.

We can now state our main result.
\begin{thm}\label{thm-main}
Let $\Phi$ be the affine IFS defined in \eqref{def-ifs}. Assume that $\#\mathcal B\ge2$ and $\ell_b\ge2$ for all $b\in\mathcal B$.  Let
$\Omega=\mathcal A^{\mathcal B^*}$ and
$\mathbb P=\pi^{\otimes\mathcal B^*}$, where
$\pi=(\pi_\gamma)_{\gamma\in\mathcal A}$ is a probability vector
on $\mathcal A$. Then, for $\mathbb P$-a.e.\ $\omega\in\Omega$,
\begin{equation}\label{eq-main-dimension}
 \dim_H K_\omega(\mathcal A)
 =\frac{1}{\log m}\log\left(
   \int_\Omega\sum_{b\in\mathcal B}\ell_b(\omega)^\lambda
   \,d\mathbb P(\omega)\right).
\end{equation}
\end{thm}

We give a brief outline of the proof. Set
$\overline L:=\int_\Omega\sum_{b\in\mathcal B}\ell_b(\omega)^\lambda
\,d\mathbb P(\omega)$.  
The weights $\ell_b(\omega)^\lambda/\overline L$ define a multiplicative cascade
on the column tree. The total masses form an $L^2$-bounded martingale
whose limit $M$ is strictly positive almost surely. We use the values of
$M$ on the subtrees to construct a random probability measure
$\mu_\omega$ on $\Sigma^{\omega}$. Integrating these measures against
$M\,d\mathbb P$ gives an invariant ergodic measure for the associated
skew product. The measures of symbolic approximate squares admit an
explicit product formula. For the upper bound, a uniform lower estimate
for the subtree masses and a comparison of averages yield the required
inequality along every admissible sequence. For the lower bound, the
ergodic theorem and the sublinear growth of $\log M$ along typical
trajectories show that the remaining terms in the formula tend to zero.
An approximate-square version of Billingsley's lemma then gives the
Hausdorff dimension.

The theorem includes the following deterministic case. If $\pi$ is concentrated at one symbol $\gamma$, then
$K_\omega(\mathcal A)=K(\mathcal A\setminus\{\gamma\})$ almost
surely. In this case, \eqref{eq-main-dimension} reduces to the classical
Bedford--McMullen formula for this deterministic subcarpet. 
Thus our result extends the fixed-symbol deletion case to independent
deletions along the column tree. The level-dependent construction of Gui
and Li  \cite{Gui2008} and the time-dependent holes of Kong, Sun and Wang \cite{Kong2026}  use different
classes of environments.

The almost sure Hausdorff dimension given by Theorem~\ref{thm-main}
depends only on the probabilities of the columns containing the
forbidden symbol. More precisely, define
$\rho_b:=\sum_{\tau(\gamma)=b}\pi_\gamma$. Indeed, 
$$
\int_\Omega\sum_{b\in\mathcal B}\ell_b(\omega)^\lambda
\,d\mathbb P(\omega)
 =\sum_{b\in\mathcal B}
   \bigl((1-\rho_b)\ell_b^\lambda+\rho_b(\ell_b-1)^\lambda\bigr).
$$

We include an example to illustrate {our} construction.
\begin{ex}\label{ex:tree-model}
Let $m=3$, $n=4$, and
\begin{equation*}
 \mathcal A=
 \{(0,0),(0,1),(0,3),(1,1),(1,2),(2,0),(2,3)\}.
\end{equation*}
Then $\mathcal B=\{0,1,2\}$, $\ell_0=3$, $\ell_1=\ell_2=2$, and
$\lambda=\log3/\log4$. The original  pattern of $K(\mathcal A)$ is shown in
Figure~\ref{fig:case-no}.

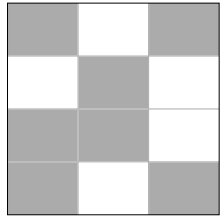
\begin{figure}[htbp]
	\captionsetup{labelfont={color=black},textfont={color=black}}
	\color{black}
	\centering
	\begin{tikzpicture}[x=0.933333333cm,y=0.700000000cm]
		\fill[black!33] (0,0) rectangle (1,1);
		\fill[black!33] (0,1) rectangle (1,2);
		\fill[black!33] (0,3) rectangle (1,4);
		\fill[black!33] (1,1) rectangle (2,2);
		\fill[black!33] (1,2) rectangle (2,3);
		\fill[black!33] (2,0) rectangle (3,1);
		\fill[black!33] (2,3) rectangle (3,4);

		\draw[white,line width=0.3pt] (1,0) -- (1,4);
		\draw[white,line width=0.3pt] (2,0) -- (2,4);
		\draw[white,line width=0.3pt] (0,1) -- (3,1);
		\draw[white,line width=0.3pt] (0,2) -- (3,2);
		\draw[white,line width=0.3pt] (0,3) -- (3,3);

		\draw[black!33,line width=0.15pt] (1,0) -- (1,4);
		\draw[black!33,line width=0.15pt] (2,0) -- (2,4);
		\draw[black!33,line width=0.15pt] (0,1) -- (3,1);
		\draw[black!33,line width=0.15pt] (0,2) -- (3,2);
		\draw[black!33,line width=0.15pt] (0,3) -- (3,3);

		\draw[black,line width=0.35pt] (0,0) rectangle (3,4);
	\end{tikzpicture}
	\caption{The original pattern of $K(\mathcal A)$.}
	\label{fig:case-no}
\end{figure}

  Choose each forbidden symbol uniformly from $\mathcal A$, independently
  at different vertices. Thus $\mathbb P(F_v=\gamma)=1/7$ for every
  $v\in\mathcal B^*$ and $\gamma\in\mathcal A$.
  Different column cylinders may have different deletion patterns,
  but the deletion at each step depends only on the preceding column word.
 
  	For example, consider a realization $\omega\in\Omega$ whose forbidden
  	symbols at the first three levels are given by
  	\[
  	\begin{aligned}
  		\text{Level }1:\quad&
  		F_{\varepsilon}(\omega)=(0,1);
  		\\[3pt]
  		\text{Level }2:\quad&
  		F_0(\omega)=(0,3),\quad
  		F_1(\omega)=(1,1),\quad
  		F_2(\omega)=(1,2);
  		\\[3pt]
  		\text{Level }3:\quad&
  		F_{00}(\omega)=(1,1),\quad
  		F_{01}(\omega)=(1,2),\quad
  		F_{02}(\omega)=(2,0),
  		\\
  		&
  		F_{10}(\omega)=(1,2),\quad
  		F_{11}(\omega)=(2,0),\quad
  		F_{12}(\omega)=(2,3),
  		\\
  		&
  		F_{20}(\omega)=(2,0),\quad
  		F_{21}(\omega)=(2,3),\quad
  		F_{22}(\omega)=(0,0).
  	\end{aligned}
  	\]
  	The corresponding first three levels are shown in
  	Figure~\ref{fig:case-tree}.
  	\par 
  
 \begin{figure}[htbp]
 	\captionsetup{labelfont={color=black},textfont={color=black}}
 	\color{black}
 	\centering
 	
 	\begin{minipage}{0.32\textwidth}
 		\centering
 		\scalebox{0.7}{\begin{tikzpicture}[x=1.366666667cm,y=1.025000000cm]
 				\fill[black!33] (0,0) rectangle (1,1);
 				\fill[black!33] (0,3) rectangle (1,4);
 				\fill[black!33] (1,1) rectangle (2,2);
 				\fill[black!33] (1,2) rectangle (2,3);
 				\fill[black!33] (2,0) rectangle (3,1);
 				\fill[black!33] (2,3) rectangle (3,4);
 				
 				\draw[white,line width=0.3pt] (1,0) -- (1,4);
 				\draw[white,line width=0.3pt] (2,0) -- (2,4);
 				\draw[white,line width=0.3pt] (0,1) -- (3,1);
 				\draw[white,line width=0.3pt] (0,2) -- (3,2);
 				\draw[white,line width=0.3pt] (0,3) -- (3,3);
 				
 				\draw[black!33,line width=0.15pt] (1,0) -- (1,4);
 				\draw[black!33,line width=0.15pt] (2,0) -- (2,4);
 				\draw[black!33,line width=0.15pt] (0,1) -- (3,1);
 				\draw[black!33,line width=0.15pt] (0,2) -- (3,2);
 				\draw[black!33,line width=0.15pt] (0,3) -- (3,3);
 				
 				\draw[black,line width=0.35pt] (0,0) rectangle (3,4);
 			\end{tikzpicture}
 		}
 		\par\smallskip {\small Level $1$}
 	\end{minipage}\hfill
 	\begin{minipage}{0.32\textwidth}
 		\centering
 		\scalebox{0.7}{\begin{tikzpicture}[x=0.455555556cm,y=0.256250000cm]
 				\fill[black!33] (0,0) rectangle (1,1);
 				\fill[black!33] (0,1) rectangle (1,2);
 				\fill[black!33] (1,1) rectangle (2,2);
 				\fill[black!33] (1,2) rectangle (2,3);
 				\fill[black!33] (2,0) rectangle (3,1);
 				\fill[black!33] (2,3) rectangle (3,4);
 				
 				\fill[black!33] (0,12) rectangle (1,13);
 				\fill[black!33] (0,13) rectangle (1,14);
 				\fill[black!33] (1,13) rectangle (2,14);
 				\fill[black!33] (1,14) rectangle (2,15);
 				\fill[black!33] (2,12) rectangle (3,13);
 				\fill[black!33] (2,15) rectangle (3,16);
 				
 				\fill[black!33] (3,4) rectangle (4,5);
 				\fill[black!33] (3,5) rectangle (4,6);
 				\fill[black!33] (3,7) rectangle (4,8);
 				\fill[black!33] (4,6) rectangle (5,7);
 				\fill[black!33] (5,4) rectangle (6,5);
 				\fill[black!33] (5,7) rectangle (6,8);
 				
 				\fill[black!33] (3,8) rectangle (4,9);
 				\fill[black!33] (3,9) rectangle (4,10);
 				\fill[black!33] (3,11) rectangle (4,12);
 				\fill[black!33] (4,10) rectangle (5,11);
 				\fill[black!33] (5,8) rectangle (6,9);
 				\fill[black!33] (5,11) rectangle (6,12);
 				
 				\fill[black!33] (6,0) rectangle (7,1);
 				\fill[black!33] (6,1) rectangle (7,2);
 				\fill[black!33] (6,3) rectangle (7,4);
 				\fill[black!33] (7,1) rectangle (8,2);
 				\fill[black!33] (8,0) rectangle (9,1);
 				\fill[black!33] (8,3) rectangle (9,4);
 				
 				\fill[black!33] (6,12) rectangle (7,13);
 				\fill[black!33] (6,13) rectangle (7,14);
 				\fill[black!33] (6,15) rectangle (7,16);
 				\fill[black!33] (7,13) rectangle (8,14);
 				\fill[black!33] (8,12) rectangle (9,13);
 				\fill[black!33] (8,15) rectangle (9,16);
 				
 				\draw[black!33,line width=0.15pt] (3,0) -- (3,16);
 				\draw[black!33,line width=0.15pt] (6,0) -- (6,16);
 				\draw[black!33,line width=0.15pt] (0,4) -- (9,4);
 				\draw[black!33,line width=0.15pt] (0,8) -- (9,8);
 				\draw[black!33,line width=0.15pt] (0,12) -- (9,12);
 				
 				\draw[black,line width=0.35pt] (0,0) rectangle (9,16);
 			\end{tikzpicture}
 		}
 		\par\smallskip {\small Level $2$}
 	\end{minipage}\hfill
 	\begin{minipage}{0.32\textwidth}
 		\centering
 		\scalebox{0.7}{\begin{tikzpicture}[x=0.151851852cm,y=0.064062500cm]
 				
 				\fill[black!33] (0,0) rectangle (1,1);
 				\fill[black!33] (0,1) rectangle (1,2);
 				\fill[black!33] (0,3) rectangle (1,4);
 				\fill[black!33] (1,2) rectangle (2,3);
 				\fill[black!33] (2,0) rectangle (3,1);
 				\fill[black!33] (2,3) rectangle (3,4);
 				
 				\fill[black!33] (0,4) rectangle (1,5);
 				\fill[black!33] (0,5) rectangle (1,6);
 				\fill[black!33] (0,7) rectangle (1,8);
 				\fill[black!33] (1,6) rectangle (2,7);
 				\fill[black!33] (2,4) rectangle (3,5);
 				\fill[black!33] (2,7) rectangle (3,8);
 				
 				\fill[black!33] (3,4) rectangle (4,5);
 				\fill[black!33] (3,5) rectangle (4,6);
 				\fill[black!33] (3,7) rectangle (4,8);
 				\fill[black!33] (4,5) rectangle (5,6);
 				\fill[black!33] (5,4) rectangle (6,5);
 				\fill[black!33] (5,7) rectangle (6,8);
 				
 				\fill[black!33] (3,8) rectangle (4,9);
 				\fill[black!33] (3,9) rectangle (4,10);
 				\fill[black!33] (3,11) rectangle (4,12);
 				\fill[black!33] (4,9) rectangle (5,10);
 				\fill[black!33] (5,8) rectangle (6,9);
 				\fill[black!33] (5,11) rectangle (6,12);
 				
 				\fill[black!33] (6,0) rectangle (7,1);
 				\fill[black!33] (6,1) rectangle (7,2);
 				\fill[black!33] (6,3) rectangle (7,4);
 				\fill[black!33] (7,1) rectangle (8,2);
 				\fill[black!33] (7,2) rectangle (8,3);
 				\fill[black!33] (8,3) rectangle (9,4);
 				
 				\fill[black!33] (6,12) rectangle (7,13);
 				\fill[black!33] (6,13) rectangle (7,14);
 				\fill[black!33] (6,15) rectangle (7,16);
 				\fill[black!33] (7,13) rectangle (8,14);
 				\fill[black!33] (7,14) rectangle (8,15);
 				\fill[black!33] (8,15) rectangle (9,16);
 				
 				\fill[black!33] (0,48) rectangle (1,49);
 				\fill[black!33] (0,49) rectangle (1,50);
 				\fill[black!33] (0,51) rectangle (1,52);
 				\fill[black!33] (1,50) rectangle (2,51);
 				\fill[black!33] (2,48) rectangle (3,49);
 				\fill[black!33] (2,51) rectangle (3,52);
 				
 				\fill[black!33] (0,52) rectangle (1,53);
 				\fill[black!33] (0,53) rectangle (1,54);
 				\fill[black!33] (0,55) rectangle (1,56);
 				\fill[black!33] (1,54) rectangle (2,55);
 				\fill[black!33] (2,52) rectangle (3,53);
 				\fill[black!33] (2,55) rectangle (3,56);
 				
 				\fill[black!33] (3,52) rectangle (4,53);
 				\fill[black!33] (3,53) rectangle (4,54);
 				\fill[black!33] (3,55) rectangle (4,56);
 				\fill[black!33] (4,53) rectangle (5,54);
 				\fill[black!33] (5,52) rectangle (6,53);
 				\fill[black!33] (5,55) rectangle (6,56);
 				
 				\fill[black!33] (3,56) rectangle (4,57);
 				\fill[black!33] (3,57) rectangle (4,58);
 				\fill[black!33] (3,59) rectangle (4,60);
 				\fill[black!33] (4,57) rectangle (5,58);
 				\fill[black!33] (5,56) rectangle (6,57);
 				\fill[black!33] (5,59) rectangle (6,60);
 				
 				\fill[black!33] (6,48) rectangle (7,49);
 				\fill[black!33] (6,49) rectangle (7,50);
 				\fill[black!33] (6,51) rectangle (7,52);
 				\fill[black!33] (7,49) rectangle (8,50);
 				\fill[black!33] (7,50) rectangle (8,51);
 				\fill[black!33] (8,51) rectangle (9,52);
 				
 				\fill[black!33] (6,60) rectangle (7,61);
 				\fill[black!33] (6,61) rectangle (7,62);
 				\fill[black!33] (6,63) rectangle (7,64);
 				\fill[black!33] (7,61) rectangle (8,62);
 				\fill[black!33] (7,62) rectangle (8,63);
 				\fill[black!33] (8,63) rectangle (9,64);
 				
 				\fill[black!33] (9,16) rectangle (10,17);
 				\fill[black!33] (9,17) rectangle (10,18);
 				\fill[black!33] (9,19) rectangle (10,20);
 				\fill[black!33] (10,17) rectangle (11,18);
 				\fill[black!33] (11,16) rectangle (12,17);
 				\fill[black!33] (11,19) rectangle (12,20);
 				
 				\fill[black!33] (9,20) rectangle (10,21);
 				\fill[black!33] (9,21) rectangle (10,22);
 				\fill[black!33] (9,23) rectangle (10,24);
 				\fill[black!33] (10,21) rectangle (11,22);
 				\fill[black!33] (11,20) rectangle (12,21);
 				\fill[black!33] (11,23) rectangle (12,24);
 				
 				\fill[black!33] (9,28) rectangle (10,29);
 				\fill[black!33] (9,29) rectangle (10,30);
 				\fill[black!33] (9,31) rectangle (10,32);
 				\fill[black!33] (10,29) rectangle (11,30);
 				\fill[black!33] (11,28) rectangle (12,29);
 				\fill[black!33] (11,31) rectangle (12,32);
 				
 				\fill[black!33] (12,24) rectangle (13,25);
 				\fill[black!33] (12,25) rectangle (13,26);
 				\fill[black!33] (12,27) rectangle (13,28);
 				\fill[black!33] (13,25) rectangle (14,26);
 				\fill[black!33] (13,26) rectangle (14,27);
 				\fill[black!33] (14,27) rectangle (15,28);
 				
 				\fill[black!33] (15,16) rectangle (16,17);
 				\fill[black!33] (15,17) rectangle (16,18);
 				\fill[black!33] (15,19) rectangle (16,20);
 				\fill[black!33] (16,17) rectangle (17,18);
 				\fill[black!33] (16,18) rectangle (17,19);
 				\fill[black!33] (17,16) rectangle (18,17);
 				
 				\fill[black!33] (15,28) rectangle (16,29);
 				\fill[black!33] (15,29) rectangle (16,30);
 				\fill[black!33] (15,31) rectangle (16,32);
 				\fill[black!33] (16,29) rectangle (17,30);
 				\fill[black!33] (16,30) rectangle (17,31);
 				\fill[black!33] (17,28) rectangle (18,29);
 				
 				\fill[black!33] (9,32) rectangle (10,33);
 				\fill[black!33] (9,33) rectangle (10,34);
 				\fill[black!33] (9,35) rectangle (10,36);
 				\fill[black!33] (10,33) rectangle (11,34);
 				\fill[black!33] (11,32) rectangle (12,33);
 				\fill[black!33] (11,35) rectangle (12,36);
 				
 				\fill[black!33] (9,36) rectangle (10,37);
 				\fill[black!33] (9,37) rectangle (10,38);
 				\fill[black!33] (9,39) rectangle (10,40);
 				\fill[black!33] (10,37) rectangle (11,38);
 				\fill[black!33] (11,36) rectangle (12,37);
 				\fill[black!33] (11,39) rectangle (12,40);
 				
 				\fill[black!33] (9,44) rectangle (10,45);
 				\fill[black!33] (9,45) rectangle (10,46);
 				\fill[black!33] (9,47) rectangle (10,48);
 				\fill[black!33] (10,45) rectangle (11,46);
 				\fill[black!33] (11,44) rectangle (12,45);
 				\fill[black!33] (11,47) rectangle (12,48);
 				
 				\fill[black!33] (12,40) rectangle (13,41);
 				\fill[black!33] (12,41) rectangle (13,42);
 				\fill[black!33] (12,43) rectangle (13,44);
 				\fill[black!33] (13,41) rectangle (14,42);
 				\fill[black!33] (13,42) rectangle (14,43);
 				\fill[black!33] (14,43) rectangle (15,44);
 				
 				\fill[black!33] (15,32) rectangle (16,33);
 				\fill[black!33] (15,33) rectangle (16,34);
 				\fill[black!33] (15,35) rectangle (16,36);
 				\fill[black!33] (16,33) rectangle (17,34);
 				\fill[black!33] (16,34) rectangle (17,35);
 				\fill[black!33] (17,32) rectangle (18,33);
 				
 				\fill[black!33] (15,44) rectangle (16,45);
 				\fill[black!33] (15,45) rectangle (16,46);
 				\fill[black!33] (15,47) rectangle (16,48);
 				\fill[black!33] (16,45) rectangle (17,46);
 				\fill[black!33] (16,46) rectangle (17,47);
 				\fill[black!33] (17,44) rectangle (18,45);
 				
 				\fill[black!33] (18,0) rectangle (19,1);
 				\fill[black!33] (18,1) rectangle (19,2);
 				\fill[black!33] (18,3) rectangle (19,4);
 				\fill[black!33] (19,1) rectangle (20,2);
 				\fill[black!33] (19,2) rectangle (20,3);
 				\fill[black!33] (20,3) rectangle (21,4);
 				
 				\fill[black!33] (18,4) rectangle (19,5);
 				\fill[black!33] (18,5) rectangle (19,6);
 				\fill[black!33] (18,7) rectangle (19,8);
 				\fill[black!33] (19,5) rectangle (20,6);
 				\fill[black!33] (19,6) rectangle (20,7);
 				\fill[black!33] (20,7) rectangle (21,8);
 				
 				\fill[black!33] (18,12) rectangle (19,13);
 				\fill[black!33] (18,13) rectangle (19,14);
 				\fill[black!33] (18,15) rectangle (19,16);
 				\fill[black!33] (19,13) rectangle (20,14);
 				\fill[black!33] (19,14) rectangle (20,15);
 				\fill[black!33] (20,15) rectangle (21,16);
 				
 				\fill[black!33] (21,4) rectangle (22,5);
 				\fill[black!33] (21,5) rectangle (22,6);
 				\fill[black!33] (21,7) rectangle (22,8);
 				\fill[black!33] (22,5) rectangle (23,6);
 				\fill[black!33] (22,6) rectangle (23,7);
 				\fill[black!33] (23,4) rectangle (24,5);
 				
 				\fill[black!33] (24,1) rectangle (25,2);
 				\fill[black!33] (24,3) rectangle (25,4);
 				\fill[black!33] (25,1) rectangle (26,2);
 				\fill[black!33] (25,2) rectangle (26,3);
 				\fill[black!33] (26,0) rectangle (27,1);
 				\fill[black!33] (26,3) rectangle (27,4);
 				
 				\fill[black!33] (24,13) rectangle (25,14);
 				\fill[black!33] (24,15) rectangle (25,16);
 				\fill[black!33] (25,13) rectangle (26,14);
 				\fill[black!33] (25,14) rectangle (26,15);
 				\fill[black!33] (26,12) rectangle (27,13);
 				\fill[black!33] (26,15) rectangle (27,16);
 				
 				\fill[black!33] (18,48) rectangle (19,49);
 				\fill[black!33] (18,49) rectangle (19,50);
 				\fill[black!33] (18,51) rectangle (19,52);
 				\fill[black!33] (19,49) rectangle (20,50);
 				\fill[black!33] (19,50) rectangle (20,51);
 				\fill[black!33] (20,51) rectangle (21,52);
 				
 				\fill[black!33] (18,52) rectangle (19,53);
 				\fill[black!33] (18,53) rectangle (19,54);
 				\fill[black!33] (18,55) rectangle (19,56);
 				\fill[black!33] (19,53) rectangle (20,54);
 				\fill[black!33] (19,54) rectangle (20,55);
 				\fill[black!33] (20,55) rectangle (21,56);
 				
 				\fill[black!33] (18,60) rectangle (19,61);
 				\fill[black!33] (18,61) rectangle (19,62);
 				\fill[black!33] (18,63) rectangle (19,64);
 				\fill[black!33] (19,61) rectangle (20,62);
 				\fill[black!33] (19,62) rectangle (20,63);
 				\fill[black!33] (20,63) rectangle (21,64);
 				
 				\fill[black!33] (21,52) rectangle (22,53);
 				\fill[black!33] (21,53) rectangle (22,54);
 				\fill[black!33] (21,55) rectangle (22,56);
 				\fill[black!33] (22,53) rectangle (23,54);
 				\fill[black!33] (22,54) rectangle (23,55);
 				\fill[black!33] (23,52) rectangle (24,53);
 				
 				\fill[black!33] (24,49) rectangle (25,50);
 				\fill[black!33] (24,51) rectangle (25,52);
 				\fill[black!33] (25,49) rectangle (26,50);
 				\fill[black!33] (25,50) rectangle (26,51);
 				\fill[black!33] (26,48) rectangle (27,49);
 				\fill[black!33] (26,51) rectangle (27,52);
 				
 				\fill[black!33] (24,61) rectangle (25,62);
 				\fill[black!33] (24,63) rectangle (25,64);
 				\fill[black!33] (25,61) rectangle (26,62);
 				\fill[black!33] (25,62) rectangle (26,63);
 				\fill[black!33] (26,60) rectangle (27,61);
 				\fill[black!33] (26,63) rectangle (27,64);
 				
 				\draw[black!33,line width=0.15pt] (9,0) -- (9,64);
 				\draw[black!33,line width=0.15pt] (18,0) -- (18,64);
 				\draw[black!33,line width=0.15pt] (0,16) -- (27,16);
 				\draw[black!33,line width=0.15pt] (0,32) -- (27,32);
 				\draw[black!33,line width=0.15pt] (0,48) -- (27,48);
 				
 				\draw[black,line width=0.35pt] (0,0) rectangle (27,64);
 			\end{tikzpicture}
 		}
 		\par\smallskip {\small Level $3$}
 	\end{minipage}
 	
 \caption{ 	The first three levels of $ K_\omega(\mathcal A) $.}
 	\label{fig:case-tree}
 \end{figure}
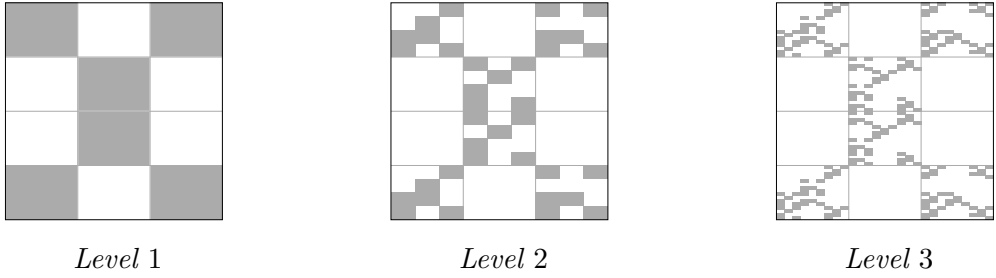 
  \FloatBarrier
  
  Under the uniform distribution on $\mathcal A$, we have
  $\rho_0=3/7$ and $\rho_1=\rho_2=2/7$.
  By Theorem~\ref{thm-main}, for $\mathbb P$-a.e.\ $\omega$,
  \[
  \dim_H K_\omega(\mathcal A)
  =\frac{1}{\log3}\log\left[
  \frac37(3\cdot2^\lambda)
  +\frac47(3^\lambda+2^\lambda+1)\right]
  \approx1.49239.
  \] 
\end{ex}

The paper is organized as follows. In Section~\ref{sec-skew-product},
we construct the cascade and the associated skew product, prove
ergodicity, and establish the estimates for the cascade limit.
Section~\ref{sec-proof} contains the proof of Theorem~\ref{thm-main}.

\section{Skew product system}\label{sec-skew-product}

In this section we construct a multiplicative cascade and a random measure
supported on the symbolic survivor set. We then prove the ergodicity of an
associated skew product and establish two estimates for the cascade limit. Throughout this section, we assume \eqref{eq-assumption}. 

Recall that $\Omega=\mathcal A^{\mathcal B^*}$ and $  \mathbb P =\bigotimes\limits_{b\in\mathcal B^*}\pi
=\pi^{\otimes\mathcal B^*}. $ 
For each $b\in\mathcal B$, define the subtree shift by
 \begin{equation}\label{eq:sigma-i}
 \widetilde\sigma_b:\Omega\to\Omega;\qquad
\omega= (\omega_v)_{v\in\mathcal B^*}\mapsto \widetilde\sigma_b(\omega)= (\omega_{bv})_{v\in\mathcal B^*}.
\end{equation} 
For $u=b_1\cdots b_k\in\mathcal B^*$, set
$\widetilde\sigma_u:=\widetilde\sigma_{b_k}\circ\cdots\circ
\widetilde\sigma_{b_1}$  with $\widetilde\sigma_\varepsilon=\operatorname{id}$.  
Then
\begin{equation}\label{eq:F-v}
 F_v(\widetilde\sigma_u\omega)=F_{uv}(\omega),
 \qquad u,v\in\mathcal B^*.
\end{equation}
 Note that for any $u\in\mathcal B^*$ the random variable  $\widetilde\sigma_u(\cdot)$ has the same distribution $\mathbb P$. 

 Recall that $\ell_b(\omega)$ is defined in \eqref{eq-liomega}. Write
\begin{equation}\label{eq-c0}
	\overline L
	:=
	\int_\Omega
	\sum_{b\in\mathcal B}
	\ell_b(\omega)^\lambda
	\,d\Ptr(\omega)
\end{equation}
and, for $b\in\mathcal B$ we  define the random variable 
\begin{equation}\label{def-wi}
	W_b(\omega)
	:=
	\frac{\ell_b(\omega)^\lambda}{\overline L}.
\end{equation}
Then 
\begin{equation}\label{eq:sum=1}
\sum_{b\in\mathcal B}\Etr W_b=1,
\end{equation}  where $\Etr W_b:=\int_{\Omega}W_b(\omega)\,d\Ptr(\omega)$. Set $M_0:=1$, and for $k\geq1$ we define
$$
	M_k(\omega)
	:=
	\sum_{b_1\cdots b_k\in\mathcal B^k}
	\prod_{r=1}^k
	W_{b_r}\bigl(
	\widetilde\sigma_{b_1\cdots b_{r-1}}\omega
	\bigr).
$$ 
{For $k\ge 0$ let
$\mathcal F_k:=\sigma(F_u:u\in\mathcal B^*,\,|u|< k)$
be the  $\sigma$-field generated by the random variables $F_u$ with $u\in\bigcup_{i=0}^{k}\mathcal B^i$. Then $\set{\mathcal F_k}_{k=0}^\infty$ forms a filtration, and}
$(M_k)_{k\geq0}$ is a nonnegative martingale with
respect to $(\mathcal F_k)_{k\geq0}$, with $\mathbb E M_k=1$.


\begin{pro}\label{pro-M-limit}
There exist a random variable $M\in L^2(\Ptr)$ and a Borel set
$\Omega_0\subseteq\Omega$ of full $\Ptr$-measure with the following properties.
	\begin{enumerate}
		\item[(i)]
		$M_k\to M$ $\Ptr$-a.e., and $\Etr M=1$.

		\item[(ii)]
		$\widetilde\sigma_u(\Omega_0)\subseteq\Omega_0$ for every
		$u\in\mathcal B^*$.

		\item[(iii)]
		For every $\omega\in\Omega_0$,
		\begin{equation*}
M(\omega)>0
		\quad\text{and}\quad
		M(\omega)
		=
		\sum_{b\in\mathcal B}
		W_b(\omega)M(\widetilde\sigma_b\omega).
\end{equation*}
	\end{enumerate}
\end{pro}

\begin{proof}

For each $u\in\mathcal B^*$, the random vector
$(W_b(\widetilde\sigma_u\omega))_{b\in\mathcal B}$ depends only on
$F_u$. Hence these random vectors are independent and identically distributed.
We use the convergence theorem for multiplicative cascades in
\cite[Section~2.3]{FalconerJin2014}. Its hypotheses
\textup{(a0)} and \textup{(a1)} in \cite[(2.1)]{FalconerJin2014}
take the form
\begin{equation}\label{eq:condition-multi-cascades}
 \Ptr\bigl(\#\{b\in\mathcal B:W_b>0\}>1\bigr)>0
 \quad\text{and}\quad
 \sum_{b\in\mathcal B}\Etr(W_b^p)<1
 \quad\text{for some }p>1.
\end{equation}
Since $\ell_b(\omega)\geq\ell_b-1\geq1$ for every $b\in\mathcal B$, {by (\ref{def-wi})
all weights $W_b, b\in\mathcal B$ are strictly positive.} Together with $\#\mathcal B\geq2$,
this proves {the first inequality in (\ref{eq:condition-multi-cascades})}.

We next verify {the second inequality in (\ref{eq:condition-multi-cascades}) with $p=2$}. By \eqref{eq-c0} {and using $\ell_b(\omega)\geq\ell_b-1\geq1$ for every $b\in\mathcal B$, it follows that} 
\begin{equation*}
 \overline L\geq(\ell_b-1)^\lambda+\#\mathcal B-1>\ell_b^\lambda 
\end{equation*}for any $b\in\mathcal B$, 
where the last inequality follows from $\#\mathcal B\geq2$ and
$x^\lambda-(x-1)^\lambda<1$ for $x\geq2$  and  $0<\lambda<1$.
Consequently, {by (\ref{def-wi}) and (\ref{eq:sum=1}) we obtain}
\begin{equation} \label{eq-cascade-L2}
\begin{split}
 \sum_{b\in\mathcal B}\Etr(W_b^2)&=\sum_{b\in\mathcal B}
 \Etr\left( W_b \frac{\ell_b(\omega)^\lambda}{\overline L}    \right)  \\  &\leq\frac{\max_{b\in\mathcal B}\ell_b^\lambda}{\overline L}
       \sum_{b\in\mathcal B}\Etr W_b
 =\frac{\max_{b\in\mathcal B}\ell_b^\lambda}{\overline L}<1.
\end{split}
\end{equation}
By \cite[Section~2.3]{FalconerJin2014}, {the inequalities in (\ref{eq:condition-multi-cascades})} imply that for {$\Ptr$-a.e. $\omega\in\Omega$} the random variable $M_k(\omega)$ converges to a non-trivial
random variable $M(\omega)$ satisfying $ 	\Etr M=1 $,  $ 	\Etr(M^2)<\infty $, and 
\begin{equation}\label{eq-smoothing}
	M(\omega)
	=
	\sum_{b\in\mathcal B}
	W_b(\omega) 
	M(\widetilde\sigma_b\omega)
	\qquad
	\Ptr\text{-a.e.}
\end{equation} 
This proves part~\textup{(i)}.

We now show that $M>0$ almost surely. {Note that} the weights in
\eqref{eq-smoothing} are strictly positive, and the random variables
$M(\widetilde\sigma_b\omega)$, $b\in\mathcal B$, are independent copies
of $M$.  Hence,
\begin{equation*}
 \Ptr(M=0)=\Ptr\left(\bigcap_{b\in\mathcal B}
       \{M(\widetilde\sigma_b\omega)=0\}\right)
     =\Ptr(M=0)^{\#\mathcal B}.
\end{equation*}
Since $\#\mathcal B\geq2$, this implies $\Ptr(M=0) \in\{0,1\}$.
The equality $\Etr M=1$ yields   $\Ptr(M=0)=0$.  So, $M(\omega)>0$ for $\Ptr$-a.e.~$\omega\in\Omega$. 

 Finally, let $\Omega'$ be a Borel set of full $\Ptr$-measure on which
$M_k\to M$, $0<M<\infty$ and \eqref{eq-smoothing} hold, and set
$\Omega_0:=\bigcap_{u\in\mathcal B^*}
\widetilde\sigma_u^{-1}(\Omega')$.
Since $\mathcal B^*$ is countable and every subtree shift preserves
$\Ptr$, the set $\Omega_0$ is Borel and has full $\Ptr$-measure.
If $\omega\in\Omega_0$ and $u,v\in\mathcal B^*$, then
$\widetilde\sigma_v(\widetilde\sigma_u\omega)
=\widetilde\sigma_{uv}\omega\in\Omega'$.
Hence $\widetilde\sigma_u\omega\in\Omega_0$, proving
part~\textup{(ii)}. Part~\textup{(iii)} follows from
$\Omega_0\subseteq\Omega'$.
\end{proof}

We use $M$ to assign transition probabilities to the admissible symbols.  Let  $ \Omega_0\subseteq \Omega $ be given {as} in Proposition \ref{pro-M-limit}.  For $\omega\in\Omega_0$ and $b\in\mathcal B$, we define
\begin{equation}\label{def-qb}
	q_b(\omega)
	:=
	\frac{W_b(\omega)M(\widetilde\sigma_b\omega)}{M(\omega)}.
\end{equation}
{Note that $W_b(\omega)=\frac{\ell_b(\omega)^\lambda}{\overline{L}}>0$ for any $b\in\mathcal B$ and any $\omega\in\Omega$.}
By Proposition~\ref{pro-M-limit}(iii),
$(q_b(\omega))_{b\in\mathcal B}$ is a strictly positive probability
vector. For $a\in\mathcal A$, define
\begin{equation}\label{eq-paom}
	p_a(\omega)
	:=
	\begin{cases}
		\displaystyle
		\frac{q_{\tau(a)}(\omega)}{\ell_{\tau(a)}(\omega)}
		=
		\frac{
			\ell_{\tau(a)}(\omega)^{\lambda-1}
			M(\widetilde\sigma_{\tau(a)}\omega)
		}{
			\overline L\cdot M(\omega)
		},
		& a\neq F_\varepsilon(\omega),
		\\[3mm]
		0,
		& a=F_\varepsilon(\omega).
	\end{cases}
\end{equation}
For each $b\in\mathcal B$, there are exactly $\ell_b(\omega)$ admissible
symbols $a$ with $\tau(a)=b$. Therefore, {by (\ref{eq-paom}) we have}
\begin{equation*}
\sum_{a\in\mathcal A}	p_a(\omega)
=
\sum_{b\in\mathcal B}   	\ell_{b}(\omega)	\frac{q_b(\omega)}{\ell_{b}(\omega)}
=1.
\end{equation*}
Thus $(p_a(\omega))_{a\in\mathcal A}$ is also a probability vector.

For a finite word $\mathbf{d}=d_1\cdots d_k\in\mathcal A^*$, let
{$[\mathbf{d}]:=\{(x_i)\in\Sigma:x_1\ldots x_k=d_1\ldots d_k\}$}  be the cylinder set in $\Sigma=\mathcal A^\N$ generated by $\mathbf d$.
For $\omega\in\Omega_0$, define a probability measure $\mu_\omega$ on $\Sigma$ by 
\begin{equation}\label{eq-mu-cylinder}
	\mu_\omega([d_1\cdots d_k])
	:=
	\prod_{r=1}^k
	p_{d_r}\bigl(
	\widetilde\sigma_{\tau(d_1\cdots d_{r-1})}\omega
	\bigr)
	\qquad\textrm{for}\quad d_1\cdots d_k\in\mathcal A^*.
\end{equation}
Proposition~\ref{pro-M-limit}(ii) ensures that all the factors in
\eqref{eq-mu-cylinder} are well defined. The cylinder weights are
consistent and $\mu_\omega$ is indeed a Borel probability measure on $\Sigma$.
Moreover, $\omega\mapsto\mu_\omega([d_1\cdots d_k])$ is measurable for
every cylinder set, so $(\mu_\omega)_{\omega\in\Omega_0}$ is a measurable
probability kernel. 

{In the next result we show that $\mu_\omega$ is supported on the symbolic survivor set $\Sigma^\omega$.}
\begin{lem}\label{lem-mu-support}
	For every $\omega \in{\Omega_0}$, the measure
	$\mu_\omega$ is supported on $\Sigma^{\omega}$.  Moreover,
	\begin{equation}\label{eq-3}
		\mu_\omega
		=
		\sum_{a\in\mathcal A}
		\mu_\omega([a])\,
		\bigl(\delta_a\otimes
		\mu_{\widetilde\sigma_{\tau(a)}\omega}\bigr).
	\end{equation}
   Here $\delta_a$ denotes the Dirac measure at $a$, and 
$$  (\delta_a\otimes\nu)
\bigl(\{b\mathbf d:\mathbf d\in E\}\bigr):=	\delta_a(\{b\})\,\nu(E) $$	 
	for every $b\in\mathcal A$ and every Borel set $E\subseteq\Sigma$.   
\end{lem}
 \begin{proof}
Fix $\omega\in\Omega_0$. By \eqref{eq-paom} and
\eqref{eq-mu-cylinder}, a cylinder $[d_1\cdots d_k]$ has positive
$\mu_\omega$-measure if and only if
$d_r\ne F_{\tau(d_1\cdots d_{r-1})}(\omega)$ for all $1\leq r\leq k$.
{Observe by (\ref{eq:symbolic-sigma-omega}) that} the complement of $\Sigma^{\omega}$ is a countable union of cylinders
on which at least one of these conditions fails. Since each such cylinder  has $\mu_\omega$-measure zero, it follows that
$\mu_\omega(\Sigma^{\omega})=1$.

To prove \eqref{eq-3}, we take  
$d_1\cdots d_k\in\mathcal A^k$ with $k\geq2$. By \eqref{eq-mu-cylinder},
\begin{align*}
 \mu_\omega([d_1\cdots d_k])
 &=p_{d_1}(\omega)\prod_{r=2}^k
   p_{d_r}\bigl(\widetilde\sigma_{\tau(d_2\cdots d_{r-1})}
            (\widetilde\sigma_{\tau(d_1)}\omega)\bigr)\\
 &=p_{d_1}(\omega)
   \mu_{\widetilde\sigma_{\tau(d_1)}\omega}([d_2\cdots d_k]),
\end{align*}
where $d_2\cdots d_{r-1}=\varepsilon$ when $r=2$.
	By the definition of $\delta_a\otimes\nu$, the measures on the two sides of \eqref{eq-3} agree on all
cylinders. Since {these} cylinders generate the Borel $\sigma$-algebra of $\Sigma$,
the desired identity follows.
\end{proof}

We next define the skew product that shifts a sequence and its
environment simultaneously. Set
\begin{equation}\label{eq-defe0}
	\mathcal E:=\{(\omega,\mathbf d )\in\Omega\times\Sigma:\mathbf d\in\Sigma^{\omega}\} \quad\text{and} \quad  \mathcal E_0:=\mathcal E\cap(\Omega_0\times\Sigma). 
\end{equation} 
Each condition in the definition of $\Sigma^{\omega}$ depends only
	on finitely many coordinates of $\omega$ and $\mathbf d=d_1d_2\ldots\in\Sigma$, and hence defines
	a Borel subset of $\Omega\times\Sigma$. Since $\mathcal E$ is the
	countable intersection of these sets, it is Borel. Since $\Omega_0$ is Borel,
$\mathcal E_0$ is also Borel. Let $\sigma:\Sigma\to\Sigma$ be the left
shift, given by $\sigma(d_1d_2\cdots)=d_2d_3\cdots$. Define  
\begin{equation}\label{def-t}
 T:\mathcal E\to\mathcal E,\qquad
 T(\omega,\mathbf d):=
 \bigl(\widetilde\sigma_{\tau(d_1)}\omega,\sigma \mathbf d\bigr).
\end{equation}
We first check that $T$ is well defined. If $(\omega,\mathbf d)\in\mathcal E$ {with $\mathbf d=d_1d_2\ldots\in\Sigma^\omega$},
then  $d_{k+2}\ne F_{\tau(d_1\cdots d_{k+1})}(\omega)$  for every
$k\geq0$. By \eqref{eq:F-v}, this is equivalent to
\begin{equation*}
 d_{k+2}\ne F_{\tau(d_2\cdots d_{k+1})}
     (\widetilde\sigma_{\tau(d_1)}\omega)\quad \text{for all }k\geq0.
\end{equation*}
Hence, $T(\omega,\mathbf d)\in\mathcal E$. Moreover,
Proposition~\ref{pro-M-limit} (ii) gives
$T(\mathcal E_0)\subseteq\mathcal E_0$. 

Define a Borel measure $\widetilde\mu$ on $\Omega\times\Sigma$ by
\begin{equation}\label{def-lmu}
	\widetilde\mu(B)
	:=
	\int_{\Omega_0}
	M(\omega)
	\int_{\Sigma}
	\mathbf1_B(\omega,\mathbf d)
	\,d\mu_\omega(\mathbf d)
	\,d\Ptr(\omega)
\end{equation} 
for every Borel set $B\subseteq\Omega\times\Sigma$.    {Since $\mu_\omega$ is a probability measure on $\Sigma^\omega\subset\Sigma$ and $\Etr M=1$ by Proposition \ref{pro-M-limit} (i)}, the measure $\widetilde\mu$ is a probability measure. 
Its environment marginal is the probability measure $\mathbb P_M$
defined by 
$$ d\mathbb P_M:=M \, d\Ptr.  $$   In particular,
$\mathbb P_M(\Omega_0)=1$.

Under the map
{$(\omega, \mathbf d)\mapsto(\omega, \tau(\mathbf d))$}, the push-forward of
$\widetilde\mu$ is the Peyri\`ere measure associated with the cascade
generated by $(W_b)_{b\in\mathcal B}$; see
\cite[(2.3) and~(2.6)]{FalconerJin2014}. The measure
$\widetilde\mu$ also records the row symbols, however, so ergodicity of this projected measure does not by itself establish ergodicity of
$\widetilde\mu$. We prove the latter below.

To prove ergodicity, we first consider the transitions of the
environment coordinate. Define the Markov operator $\mathcal L$ on
bounded measurable functions on $\Omega_0$ by
$$
	\mathcal Lh(\omega)
	:=
	\sum_{b\in\mathcal B}
	q_b(\omega)h(\widetilde\sigma_b\omega).
$$
\begin{lem}\label{lem-kf}
	The probability measure $\mathbb P_M$ is stationary for $\mathcal L$.
	Moreover, every bounded $\mathcal L$-harmonic function is
	$\mathbb P_M$-a.e.\ constant.
\end{lem}
\begin{proof}

We first prove the $\mathcal L$-stationarity of $\mathbb P_M$. {Note that $\mathbb P_M(\Omega_0)=1$.}  Let $g$ be a bounded measurable function on
$\Omega_0$. Then by \eqref{def-qb} { and using $d\mathbb P_M=M d\mathbb P$},
\begin{equation*}
 \int_{\Omega_0}\mathcal Lg\,d\mathbb P_M
 =\sum_{b\in\mathcal B}\Etr\left[
   {W_b\cdot (M\cdot g)\circ\widetilde\sigma_b }\right].
\end{equation*}
For each $b\in\mathcal B$, {by (\ref{def-wi})} the weight $W_b$ depends only on the root
variable $F_\varepsilon$, whereas
$(M\cdot g)\circ\widetilde\sigma_b$ depends only on the subtree rooted at $b$.
These random variables are independent. Since
$\widetilde\sigma_b\omega$ has the distribution $\Ptr$ and
$\sum_b\Etr W_b=1$, we obtain
\begin{equation*}
 \int_{\Omega_0}\mathcal Lg\,d\mathbb P_M
 =\sum_{b\in\mathcal B}(\Etr W_b)\,\Etr(M\cdot g)
 =\Etr(M\cdot g)
 =\int_{\Omega_0}g\,d\mathbb P_M.
\end{equation*}
Thus $\mathbb P_M$ is stationary for $\mathcal L$.

Now let $h$ be a bounded real-valued function satisfying
$\mathcal Lh=h$ $\mathbb P_M$-a.e. By Jensen's inequality,
\begin{equation}\label{eq-lh}
 (\mathcal Lh)^2={\left(\sum_{b\in\mathcal B}q_b(\omega)h(\tilde\sigma_b\omega)\right)^2\leq \sum_{b\in\mathcal B}q_b(\omega)h^2(\tilde\sigma_b\omega)}=\mathcal L(h^2).
\end{equation}
Using  $\mathcal Lh=h$ and the $\mathcal L$-stationarity of $\mathbb P_M$, we obtain
\begin{equation*}
 \int_{\Omega_0}(\mathcal Lh)^2\,d\mathbb P_M
 =\int_{\Omega_0}h^2\,d\mathbb P_M
 =\int_{\Omega_0}\mathcal L(h^2)\,d\mathbb P_M.
\end{equation*}
Hence equality holds in \eqref{eq-lh} almost everywhere.  Since all the weights
$q_b$ are strictly positive, the equality in Jensen's inequality implies
that $h(\widetilde\sigma_b\omega)$ has the same value for every
$b\in\mathcal B$. Since $h=\mathcal Lh$, this common value is
$h(\omega)$. Thus,
\begin{equation*}
h(\widetilde\sigma_b\omega)=h(\omega) 
 \qquad\forall b\in\mathcal B
\end{equation*}
for $\mathbb P_M$-a.e.\ $\omega\in\Omega_0$.

By Proposition~\ref{pro-M-limit}, we have $M(\omega)>0$ $\Ptr$-a.e.~$\omega\in\Omega_0$.
Then $\mathbb P_M$ and $\Ptr$ are equivalent, and the preceding
identities also hold $\Ptr$-a.e.~$\omega\in\Omega_0$. {So, under $\Ptr$ for distinct $b,c\in\mathcal B$,} the random variables
$h\circ\widetilde\sigma_b$ and $h\circ\widetilde\sigma_c$ are independent,
and each has the same distribution as $h$. Since both are equal to
$h$ $\Ptr$-a.e., we obtain
\begin{equation*}
	0=\Etr\bigl[(h\circ\widetilde\sigma_b
	-h\circ\widetilde\sigma_c)^2\bigr]
	=2\bigl(\Etr(h^2)-(\Etr h)^2\bigr).
\end{equation*}
{Again, by using Jensen's inequality we conclude that} $h$ is a constant $\Ptr$-a.e., and hence also
$\mathbb P_M$-a.e.  
\end{proof}

\begin{pro}\label{pro-ergodic}
The measure $\widetilde\mu$ is concentrated on $\mathcal E_0$, and is
$T$-invariant and ergodic.
\end{pro}
\begin{proof}
By {(\ref{def-lmu}) and} Lemma~\ref{lem-mu-support},
\begin{equation*}
 \widetilde\mu\bigl((\Omega\times\Sigma)\setminus\mathcal E_0\bigr)
 =\int_{\Omega_0}
 M(\omega)\mu_\omega(\Sigma\setminus\Sigma^{\omega})\,d\Ptr(\omega)=0.
\end{equation*}
Thus $\widetilde\mu$ is concentrated on $\mathcal E_0$.

We next prove that $\widetilde\mu$ is $T$-invariant. Let $f$ be a bounded
measurable function on $\mathcal E_0$, extended by zero to
$\Omega_0\times\Sigma$. We define
$$ g_f(\omega):=\int_\Sigma f(\omega,\mathbf d)\,d\mu_\omega(\mathbf d).  $$ 
The function $g_f$ is bounded and measurable, and
$\int f\,d\widetilde\mu=\int_{\Omega_0}g_f\,d\mathbb P_M$.
By \eqref{eq-3}, {(\ref{def-t})} and Lemma~\ref{lem-kf},
\begin{align*}
 \int f\circ T\,d\widetilde\mu&{=\int_{\Omega_0}\int_{\Sigma}f(\tilde\sigma_{\tau(d_1)}\omega, \sigma\mathbf d)d\mu_\omega(\mathbf d)d\mathbb P_M(\omega)}\\
 &=\int_{\Omega_0}\sum_{a\in\mathcal A}p_a(\omega)
   \int_\Sigma f(\widetilde\sigma_{\tau(a)}\omega,\mathbf d)
       \,d\mu_{\widetilde\sigma_{\tau(a)}\omega}(\mathbf d)
       \,d\mathbb P_M(\omega)\\
 &=\int_{\Omega_0}\sum_{b\in\mathcal B}
   q_b(\omega)g_f(\widetilde\sigma_b\omega)\,d\mathbb P_M(\omega)\\
 &=\int_{\Omega_0}\mathcal Lg_f\,d\mathbb P_M
 =\int_{\Omega_0}g_f \, d\mathbb P_M
 =\int f\,d\widetilde\mu.
\end{align*}
This proves the $T$-invariance of $\widetilde{\mu}$.

It remains to prove the ergodicity. {Let $A\subseteq\mathcal E_0$ be a Borel set satisfying  $\widetilde\mu(T^{-1}A \triangle A) = 0$. Then
 $\mathbf1_A(\omega,\mathbf d)=\mathbf1_A\circ T(\omega, \mathbf d)$ for $\widetilde\mu$-a.e.~$(\omega,\mathbf d)\in\mathcal E_0$.}
Define
$$g_A(\omega):=\int_\Sigma\mathbf1_A(\omega,\mathbf d)\, d\mu_\omega(\mathbf d)\qquad \textrm{for}\quad\omega\in\Omega_0.   $$  
Using \eqref{eq-3} again {and $\mathbf1_A=\mathbf1_A\circ T$ for $\widetilde\mu$-a.e.}, we obtain
\begin{equation*}
 g_A(\omega)
 =\int_\Sigma\mathbf1_A(T(\omega,\mathbf d))\,d\mu_\omega(\mathbf d)
 {=\sum_{b\in\mathcal B}q_b(\omega)\int_{\Sigma}\mathbf 1_A(\tilde\sigma_b\omega,\mathbf d)d\mu_{\tilde\sigma_b\omega}(\mathbf d)}=\mathcal Lg_A(\omega)
 \qquad\mathbb P_M\text{-a.e.}
\end{equation*}
By Lemma~\ref{lem-kf}, there is a constant $c\in[0,1]$ such that
\begin{equation}\label{eq-gC0}
 g_A(\omega)=c\qquad\mathbb P_M\text{-a.e. } \omega\in\Omega_0.
\end{equation}
We will show that $\mathbf1_A=c$ almost everywhere.

For $n\geq1$, write
${h_n}(\omega,\mathbf d):=\widetilde\sigma_{\tau(d_1\cdots d_n)}\omega$ for $(\omega,\mathbf d)\in\mathcal E_0$,
and let $\mathcal H_n$ be the $\sigma$-algebra on $\mathcal E_0$
generated by the environment coordinate and the first $n$ symbols of
$\mathbf d$. Then $(\mathcal H_n)$ is increasing and generates the Borel
$\sigma$-algebra of $\mathcal E_0$. The function $g_A\circ{h_n}$ is
$\mathcal H_n$-measurable. Moreover, ${h_n}$ is the environment
coordinate of $T^n(\omega,\mathbf d)$. Since $\widetilde\mu$ is $T$-invariant
and its environment marginal is $\mathbb P_M$, we have
$\widetilde\mu\circ{h_n^{-1}}=\mathbb P_M$. Hence \eqref{eq-gC0} gives
\begin{equation}\label{eq-YnC}
 g_A\circ{h_n}=c\qquad\widetilde\mu\text{-a.e.}
\end{equation}

We claim that
\begin{equation}\label{eq-cond}
 \mathbb E_{\widetilde\mu}(\mathbf1_A\mid\mathcal H_n)
 =g_A\circ{h_n}
 \qquad\widetilde\mu\text{-a.e.}
\end{equation}
To prove the claim, it suffices to test the identity on sets of the form
$B=\mathcal E_0\cap(F\times[a_1\cdots a_n])$, where
$F\subseteq\Omega_0$ is Borel {and $a_1\ldots a_n\in\mathcal A^n$}. These sets form a generating
$\pi$-system for $\mathcal H_n$.  
Since $\mathbf1_A=\mathbf1_A\circ T^n$ almost everywhere, the iterated
form of \eqref{eq-3} gives
\begin{align*}
 \int_B\mathbf1_A\,d\widetilde\mu
 &=\int_B\mathbf1_A\circ T^n\,d\widetilde\mu\\
 &=\int_F\mu_\omega([a_1\cdots a_n])
       g_A(\widetilde\sigma_{\tau(a_1\ldots a_n)}\omega)\,d\mathbb P_M(\omega)\\
&=	\int_F \int_{[a_1\cdots a_n]} \mathbf 1_{[a_1\cdots a_n]}(\mathbf d) 
g_A(\widetilde\sigma_{\tau(a_1\ldots a_n)}\omega)
\, d\mu_\omega(\mathbf d) d\mathbb P_M(\omega)\\
 &=\int_B g_A\circ{h_n}\,d\widetilde\mu.
\end{align*}
This proves \eqref{eq-cond}.

By the martingale convergence theorem,
$\mathbb E_{\widetilde\mu}(\mathbf1_A\mid\mathcal H_n)
\to\mathbf1_A$ $\widetilde\mu$-a.e.
Combining this with \eqref{eq-YnC} and \eqref{eq-cond}, we obtain
$\mathbf1_A=c$ almost everywhere. Thus $c\in\{0,1\}$ and
$\widetilde\mu(A)\in\{0,1\}$, proving ergodicity.
\end{proof}

We now estimate the values of $M$ on subtrees. The first estimate
controls $\log M$ along $\widetilde\mu$-almost every orbit and will
be used to prove the lower bound for the Hausdorff dimension of a random survivor set {of $K_\omega(\mathcal A)$}.
 \begin{lem}\label{lem-M0}
	For $\widetilde\mu$-a.e. $(\omega,\mathbf d)\in\mathcal E_0$,
	\begin{equation*}
\lim_{k\to\infty}
	\frac{
		\log M\bigl(
		\widetilde\sigma_{\tau(d_1\cdots d_k)}\omega
		\bigr)
	}{k}
	=0.
\end{equation*}
\end{lem}
\begin{proof}
	By Proposition~\ref{pro-M-limit}, $M>0$ almost surely and
$M\in L^2(\mathbb P)$. Define
$Y(\omega,\mathbf d):=|\log M(\omega)|$ {for $(\omega,\mathbf d)\in\mathcal E_0$.} 
Since the environment marginal of $\widetilde\mu$ is $\mathbb P_M$,
\begin{equation*}
 \int_{\mathcal E_0}Y\,d\widetilde\mu{=\int_{\Omega_0}|\log M|\,d\mathbb P_M}
 =\int_{\Omega_0} M|\log M|\,d\Ptr
 \leq e^{-1}+\int_{\{M\geq1\}}M^2\,d\Ptr<\infty.
\end{equation*}
Here we used $t|\log t|\leq e^{-1}$ for $0<t\leq1$,
$t\log t\leq t^2$ for $t\geq1$, and
$M\in L^2(\Ptr)$ from Proposition~\ref{pro-M-limit}.
Fix $\delta>0$. By Proposition~\ref{pro-ergodic},
$\widetilde\mu$ is $T$-invariant. Hence,
\begin{align*}
 \sum_{k=1}^\infty\widetilde\mu(Y\circ T^k>\delta k)
 &=\sum_{k=1}^\infty\widetilde\mu(Y>\delta k)\\
 &=\int_{\mathcal E_0}\sum_{k=1}^\infty
       \mathbf1_{\{Y>\delta k\}}\,d\widetilde\mu
 \leq\frac1\delta\int_{\mathcal E_0}Y\,d\widetilde\mu<\infty.
\end{align*}
The Borel--Cantelli lemma therefore implies that, for
$\widetilde\mu$-a.e.\ $(\omega,\mathbf d)\in\mathcal E_0$, we have
$Y(T^k(\omega,\mathbf d))\leq\delta k$ for all sufficiently large $k$.  Taking the intersection of these full-measure sets over positive
rational $\delta$ yields $Y\circ T^k/k\to0$ almost surely.
Since
$Y\circ T^k(\omega,\mathbf d)
=|\log M(\widetilde\sigma_{\tau(d_1\cdots d_k)}\omega)|$,
the lemma follows.
\end{proof}

To prove the upper bound for the Hausdorff dimension of the survivor
set $K_\omega(\mathcal A)$, we need a lower bound for $M(\widetilde\sigma_u\omega)$
that is uniform over all words of a given length.
\begin{lem}\label{lem-uniform}
	For $\Ptr$-a.e.\ $\omega\in{\Omega_0}$,
	\begin{equation}\label{eq-uniform}
		\limsup_{k\to\infty}
		\frac{1}{k}
		\max_{u\in\mathcal B^k}
		\bigl(-\log M(\widetilde\sigma_u\omega)\bigr)
		\leq 0.
	\end{equation}
\end{lem}
\begin{proof}
For $t>0$ let
\[G(t):=\Ptr(M\leq t).\]
 Since $M>0$ almost surely,
$G(t)\to0$ as $t\downarrow0$. If $\omega\in\Omega_0$ and $M(\omega)\leq t$, then  by
Proposition~\ref{pro-M-limit}(iii) and \eqref{def-wi} it follows that
\begin{equation*}
 \ell_b(\omega)^\lambda M(\widetilde\sigma_b\omega){=\overline L\cdot W_b(\omega) M(\widetilde{\sigma_b}\omega)}
 \leq \overline L\cdot M(\omega)\leq \overline L\cdot t   \qquad \forall b\in\mathcal B.
\end{equation*}
Since $\ell_b(\omega)^\lambda\geq1$, we have
$M(\widetilde\sigma_b\omega)\leq \overline Lt$ for every $b\in\mathcal B$.
{Note that} the variables $ M(\widetilde\sigma_b\omega), $
$b\in\mathcal B$,  are
independent copies of $M$. Hence,
\begin{equation}\label{eq-G-recursion}
	G(t)\le  \prod_{b\in \mathcal B}  \Ptr(\set{\omega\in\Omega_0:M(\widetilde\sigma_b\omega)\leq \overline L\cdot t}){=\prod_{b\in \mathcal B}  \Ptr(M\leq \overline L\cdot t)} \leq G(\overline L\cdot t)^{\#\mathcal B}.
\end{equation}

Suppose first that $G(s)=0$ for some $s>0$. Since each subtree
shift preserves $\Ptr$ and $\mathcal B^*$ is countable, for
$\Ptr$-a.e.\ $\omega$ we have
$M(\widetilde\sigma_u\omega)>s$ for all $u\in\mathcal B^*$.
The bound \eqref{eq-uniform} follows immediately. It remains to consider the case where $G(t)>0$ for all $t>0$.
Choose $t_0\in(0,1)$ such that $0<G(t_0)<1$. Since
$\ell_b(\omega)\geq1$ for every $b\in\mathcal B$, {by (\ref{eq-assumption}) and (\ref{eq-c0})} we have
$\overline L\geq\#\mathcal B\geq2$. Iterating \eqref{eq-G-recursion},
we obtain
\begin{equation}\label{eq-G-iterate}
	G(t_0 \overline L^{-r})\leq G(t_0)^{(\#\mathcal B)^r},
	\qquad r=0,1,2,\ldots.
\end{equation} 

 Fix $\delta>0$ and let 
$r_k:=\lfloor\delta k/(2\log \overline L)\rfloor$ for $ k\in \N $.
For all sufficiently large $k$ we have
$e^{-\delta k}\leq t_0 \overline L^{-r_k}$.  Hence, by \eqref{eq-G-iterate} and for all sufficiently large $k$ it follows that 
\begin{align*}
 \Ptr\left(\min_{u\in\mathcal B^k}
     M(\widetilde\sigma_u\omega)\leq e^{-\delta k}\right)&\leq\sum_{u\in\mathcal B^k}
     \Ptr\left(M(\widetilde\sigma_u\omega)
                    \leq t_0\overline L^{-r_k}\right)\\
 &=(\#\mathcal B)^kG(t_0\overline L^{-r_k})
 \leq(\#\mathcal B)^kG(t_0)^{(\#\mathcal B)^{r_k}}.
\end{align*}
{Since $G(t_0)\in(0,1)$, 
the last expression is summable in $k$.} 
By the Borel--Cantelli lemma, for $\Ptr$-a.e.\ $\omega$ and all
sufficiently large $k$, we have
$M(\widetilde\sigma_u\omega)>e^{-\delta k}$ for every
$u\in\mathcal B^k$. Consequently,  for $\Ptr$-a.e.\ $\omega$ we have 
\begin{equation*}
 \limsup_{k\to\infty}\frac1k
 \max_{u\in\mathcal B^k}
   \bigl(-\log M(\widetilde\sigma_u\omega)\bigr)\leq\delta.
\end{equation*}
Taking the intersection over positive rational $\delta$ proves
\eqref{eq-uniform}.
\end{proof}

\section{Proof of Theorem~\ref{thm-main}}\label{sec-proof}

In this section, we use the ergodic skew product constructed in
Section~\ref{sec-skew-product} and the estimates for $M$ obtained
there to prove Theorem~\ref{thm-main}.

For $k\geq1$, write $\theta_k:=\lfloor\lambda k\rfloor$.
For $\mathbf d=(d_r)_{r\geq1}\in\Sigma$, define the approximate square by
{\begin{equation}\label{eq-Qkd}
		Q_k(\mathbf d)
		=\Pi\left(\bigl\{(z_r)_{r=1}^{\infty}\in\Sigma:  		 
		z_1\ldots z_{\theta_k}=d_1\ldots d_{\theta_k};~ \tau(z_{\theta_k+1}\ldots z_{k})=\tau(d_{\theta_k+1}\ldots d_k) \bigr\}\right),
\end{equation}}
where $\Pi$ is the coding map defined in \eqref{eq}. Thus $Q_k(\mathbf d)$ is contained in a rectangle whose sides have length $ m^{-k} $ and $ n^{-\theta_k} .$ Since
$m^{-k}\leq n^{-\theta_k}<n\,m^{-k}$ {for any $k\in\N$}, its diameter is uniformly
comparable to $m^{-k}$. Moreover, distinct approximate squares of the
same level have disjoint interiors.

 The following lemma is the standard
approximate-square version of Billingsley's lemma; cf.~\cite{Billingsley1965}.

\begin{lem}\label{lem-bil}
	Let $\nu$ be a Borel probability measure
	supported on $\Sigma^{\omega}$  with $\omega\in\Omega$, and let $s\geq0$. If
	\begin{equation*}
		\liminf_{k\to\infty}
		\frac{\log\nu(\Pi^{-1}(Q_k(\mathbf d)))}{-k\log m}\geq s
		\qquad\text{for }\nu\text{-a.e. }\mathbf d\in\Sigma^{\omega},
	\end{equation*}
	then $\dim_H K_\omega(\mathcal A)\geq s$. If
	\begin{equation*}
		\liminf_{k\to\infty}
		\frac{\log\nu(\Pi^{-1}(Q_k(\mathbf d)))}{-k\log m}\leq s
		\qquad\text{for every }\mathbf d\in\Sigma^{\omega},
	\end{equation*}
	then $\dim_H K_\omega(\mathcal A)\leq s$.
\end{lem}
	Let $\Omega_0\subseteq\Omega$ be given as in Proposition \ref{pro-M-limit}. For $\omega\in\Omega_0$, let   $\mu_\omega$ be the  probability measure on $\Sigma$ defined  in  
  \eqref{eq-mu-cylinder}.  Furthermore, let $ \widetilde{\mu} $ be {the probability measure supported on $\Omega\times\Sigma$  defined as in \eqref{def-lmu}. Note by Proposition \ref{pro-ergodic} that $\widetilde{\mu}$ is supported on the skew product $\mathcal E_0$ defined as in \eqref{eq-defe0}.} 
 
 \begin{lem}\label{lem-G}
 	Set
 	\begin{equation}\label{eq-grid-boundary}
 		G:=
 		\bigcup_{k\geq1}
 		\left(
 		\bigcup_{p=0}^{m^k}
 		\left(\left\{\frac{p}{m^k}\right\}\times[0,1]\right)
 		\;\cup\;
 		\bigcup_{q=0}^{n^{\theta_k}}
 		\left([0,1]\times
 		\left\{\frac{q}{n^{\theta_k}}\right\}\right)
 		\right).
 	\end{equation}
 	Then $\mu_\omega(\Pi^{-1}(G))=0$ for
 	$\mathbb P$-a.e.\ $\omega\in\Omega_0$.
 \end{lem}
  
 	\begin{proof}{For an element in $\Sigma$ we write  $\mathbf d=(d_r)_{r\geq1}\in\Sigma$ with
 	each $d_r=(d^{(1)}_r,d^{(2)}_r)\in\mathcal A$.}
 		For $i\in\{0,m-1\}$ and $j\in\{0,n-1\}$, we define
 		{\[
 		\begin{aligned}
 			E_i
 			&:=
 			\bigcup_{k=1}^\infty\left\{(\omega,\mathbf d)\in\mathcal E_0:
 			{d^{(1)}_r}=i \text{ for all }r\ge k\right\},\\
 			F_j
 			&:=
 			\bigcup_{k=1}^\infty\left\{(\omega,\mathbf d)\in\mathcal E_0:
 			{d^{(2)}_r}=j \text{ for all  }r\ge k\right\}.
 		\end{aligned}
 		\]}
 		
 		Fix $i\in\{0,m-1\}$ and $j\in\{0,n-1\}$.
 		By \eqref{eq-assumption}, we have $\#\mathcal B\geq2$
 		and $\ell_b(\omega)\geq1$ for every $b\in\mathcal B$
 		and $\omega\in\Omega_0$.
 		Moreover, a column not containing the forbidden symbol
 		retains at least two distinct row digits.
 		Since $p_a(\omega)>0$ for every admissible symbol $a\in\mathcal A$,
 		we have, for every $\omega\in\Omega_0$,
 		\begin{equation}\label{eq-p1}
 			\sum_{u:(i,u)\in\mathcal A}p_{(i,u)}(\omega)<1,
 			\qquad
 			\sum_{b:(b,j)\in\mathcal A}p_{(b,j)}(\omega)<1.
 		\end{equation}
 		{By (\ref{def-lmu})} it follows that
 		\begin{align*}
 			\alpha_i
 			&:=
 			\widetilde\mu\bigl(
 			\{(\omega,\mathbf d)\in\mathcal E_0:{d^{(1)}_1}=i\}
 			\bigr)\\
 			&=
 			\int_{\Omega_0}
 			\mu_\omega(\{\mathbf d:{d^{(1)}_1}=i\})
 			M(\omega)\,d\mathbb P(\omega)\\
 			&=
 			\sum_{u:(i,u)\in\mathcal A}
 			\int_{\Omega_0}
 			p_{(i,u)}(\omega)M(\omega)\,d\mathbb P(\omega)<1,
 		\end{align*}
 		where the last inequality follows from \eqref{eq-p1}
 		and the identity $\mathbb E M=1$ in
 		Proposition~\ref{pro-M-limit}.
 		By a similar argument, we have
 		$
 		\beta_j
 		:=
 		\widetilde\mu\bigl(
 		\{(\omega,\mathbf d)\in\mathcal E_0:{d^{(2)}_1}=j\}
 		\bigr)<1.
 		$
 		
 		Define $f_i:\mathcal E_0\to\{0,1\}$ by
 		\[
 		f_i(\omega,\mathbf d)
 		:=
 		\mathbf{1}_{\{{d^{(1)}_1}=i\}}(\omega,\mathbf d) 
 		\qquad\textrm{for }
 		(\omega,\mathbf d)\in\mathcal E_0.
 		\]
 		{Note by Proposition~\ref{pro-ergodic} that $\widetilde\mu$ is
 		$T$-invariant and ergodic on $\mathcal E_0$.}
 		Applying Birkhoff's ergodic theorem to $f_i$, we obtain
 		\begin{align*}
 			\lim_{N\to\infty}\frac1N
 			\sum_{r=1}^N
 			\mathbf{1}_{\{{d^{(1)}_r}=i\}}(\omega,\mathbf d)
 			&=
 			\lim_{N\to\infty}\frac1N
 			\sum_{r=1}^N
 			f_i\bigl(T^{r-1}(\omega,\mathbf d)\bigr)\\
 			&=
 			\int_{\mathcal E_0}f_i\,d\widetilde\mu =
 			\widetilde\mu\bigl(
 			\{(\omega,\mathbf d)\in\mathcal E_0:{d^{(1)}_1}=i\}
 			\bigr)
 			=\alpha_i<1
 		\end{align*}
 		for $\widetilde\mu$-a.e.\
 		$(\omega,\mathbf d)\in\mathcal E_0$.
 		On the other hand,
 		\[
 		E_i\subseteq
 		\left\{(\omega,\mathbf d)\in\mathcal E_0:
 		\lim_{N\to\infty}\frac1N
 		\sum_{r=1}^N
 		\mathbf{1}_{\{{d^{(1)}_r}=i\}}(\omega,\mathbf d)=1
 		\right\}.
 		\]
 		It follows that $\widetilde\mu(E_i)=0$.
 		By a similar argument, we have $\widetilde\mu(F_j)=0$.
 		
 		Note that every base-$m$ expansion of $p/m^k$
 		is eventually constant with value $0$ or $m-1$,
 		and the analogous statement holds in base $n$.
 		Thus
 		\[
 		\{(\omega,\mathbf d)\in\mathcal E_0:\Pi(\mathbf d)\in G\}
 		\subseteq
 		\bigcup_{i\in\{0,m-1\}}E_i
 		\;\cup\;
 		\bigcup_{j\in\{0,n-1\}}F_j.
 		\]
 Using this inclusion and the definition of $\widetilde\mu$, we obtain
 		\[
 		\begin{aligned}
 			 \int_{\Omega_0}
 			\mu_\omega(\Pi^{-1}(G))M(\omega)\,d\mathbb P(\omega) 
 			&=
 			\widetilde\mu\bigl(
 			\{(\omega,\mathbf d)\in\mathcal E_0:
 			\Pi(\mathbf d)\in G\}
 			\bigr)\\
 			& \leq
 			\sum_{i\in\{0,m-1\}}\widetilde\mu(E_i)
 			+
 			\sum_{j\in\{0,n-1\}}\widetilde\mu(F_j)
 			=0.
 		\end{aligned}
 		\]
 Combining this with the fact that $M>0$ on $\Omega_0$
 by Proposition~\ref{pro-M-limit}, we conclude that
 $\mu_\omega(\Pi^{-1}(G))=0$ for
 $\mathbb P$-a.e.\ $\omega\in\Omega_0$.
 	\end{proof}

 The following lemma was established in \cite[Lemma 4.1]{Kenyon1996}.
\begin{lem}\label{lem-limsup}
  If $f:\N\to\R$ satisfies $\sup_{k\in\N}|f(k+1)-f(k)|<+\infty$, then for any $\alpha>0$ we have 
  \[
  \limsup_{k\to\infty}\frac{1}{k}\left(\alpha f(k)-f(\lfloor\alpha k\rfloor)\right)\ge 0.
  \]
\end{lem}  
\begin{proof}[Proof of Theorem \ref{thm-main}] 
	By Lemmas~\ref{lem-uniform} and~\ref{lem-G},
	choose a $\mathbb P$-full set $\Omega_1\subseteq\Omega_0$
	such that, for every $\omega\in\Omega_1$,
	\eqref{eq-uniform} holds and
	$\mu_\omega(\Pi^{-1}(G))=0$.
	
	Fix $\omega\in\Omega_1$ and let $\mu_\omega$ be the measure
	defined as in \eqref{eq-mu-cylinder}.
	By Lemma~\ref{lem-mu-support}, $\mu_\omega$ is supported on
	$\Sigma^{\omega}$, and hence $\mu_\omega\Pi^{-1}$ is
	supported on $K_\omega(\A){=\Pi(\Sigma^\omega)}$.
	Moreover, every admissible cylinder set has positive
	$\mu_\omega$-measure, and therefore
	$\operatorname{supp}(\mu_\omega\Pi^{-1})=K_\omega(\A)$.
	
	Let $\mathbf d=(d_r)_{r=1}^{\infty}\in\Sigma^{\omega}$ with each $d_r=(i_r,j_r)\in\mathcal A$.
	For $k\in\mathbb N$, set $\theta_k:=\lfloor\lambda k\rfloor$
	and define
	\[
	C_k(\mathbf d)
	:=
	\bigcup_{\substack{
			a_{\theta_k+1}\cdots a_k\in\mathcal A^{k-\theta_k}:\\
			\tau(a_r)=i_r,\ \theta_k<r\leq k
	}}
	[d_1\cdots d_{\theta_k}a_{\theta_k+1}\cdots a_k].
	\]
	It is clear that $Q_k(\mathbf d)=\Pi(C_k(\mathbf d))$.
	On the other hand,
	\[
	\Pi^{-1}\bigl(Q_k(\mathbf d)\setminus G\bigr)
	\subseteq C_k(\mathbf d).
	\]
	Indeed, otherwise there would exist
	$\mathbf i\in\mathcal A^{\mathbb N}\setminus C_k(\mathbf d)$
	and $\mathbf j\in C_k(\mathbf d)$ such that
	$\Pi(\mathbf i)=\Pi(\mathbf j)\notin G$,
	contradicting the uniqueness of the coordinate expansions
	outside $G$.
	Since $\mu_\omega(\Pi^{-1}(G))=0$,
	it follows from \eqref{eq-mu-cylinder} that
	\begin{equation}\label{eq-muQ}
		\begin{aligned}
			\mu_\omega\Pi^{-1}(Q_k(\mathbf d))
			&=\mu_\omega(C_k(\mathbf d))\\
			&=
			\prod_{r=1}^{\theta_k}
			p_{d_r}\bigl(
			\widetilde\sigma_{i_1\cdots i_{r-1}}\omega
			\bigr)
			\prod_{r=\theta_k+1}^{k}
			\left(
			\sum_{\substack{a\in\mathcal A\\\tau(a)=i_r}}
			p_a\bigl(
			\widetilde\sigma_{i_1\cdots i_{r-1}}\omega
			\bigr)
			\right)\\
			&=
			\prod_{r=1}^{\theta_k}
			p_{d_r}\bigl(
			\widetilde\sigma_{i_1\cdots i_{r-1}}\omega
			\bigr)
			\prod_{r=\theta_k+1}^{k}
			q_{i_r}\bigl(
			\widetilde\sigma_{i_1\cdots i_{r-1}}\omega
			\bigr),
		\end{aligned}
	\end{equation}
	where $q_b$ is defined in \eqref{def-qb}, and the last equality
	follows from
	$\sum_{\tau(a)=b}p_a(\omega)=q_b(\omega)$.   Substituting
	\eqref{eq-paom} and \eqref{def-qb}  in \eqref{eq-muQ} and cancelling the successive
	factors involving $M$, we obtain  
	\begin{equation}\label{eq-muQ1}
	 	\begin{aligned}
	 		\mu_\omega\Pi^{-1}(Q_k(\mathbf d))	&=	\prod_{r=1}^{\theta_k} \frac{		\ell_{i_r}(\widetilde\sigma_{i_1\cdots i_{r-1}}\omega)^{\lambda-1} M(\widetilde\sigma_{i_1\cdots i_{r}}\omega)}{{\overline{L}} \cdot M(\widetilde\sigma_{i_1\cdots i_{r-1}}\omega)} \\
&\qquad\cdot  	\prod_{r=\theta_k+1}^{k} \frac{		\ell_{i_r}(\widetilde\sigma_{i_1\cdots i_{r-1}}\omega)^{\lambda} M(\widetilde\sigma_{i_1\cdots i_{r}}\omega)}{{\overline{L}} \cdot  M(\widetilde\sigma_{i_1\cdots i_{r-1}}\omega)}  \\
	 		&= \frac{M(\widetilde\sigma_{i_1\cdots i_{k}}\omega)}{{\overline{L}^k}\cdot M(\omega)}	\cdot\frac{
	 			\prod_{r=1}^{k}
	 			\ell_{i_r}(\widetilde\sigma_{i_1\cdots i_{r-1}}\omega)^{\lambda}
	 		}{
	 			\prod_{r=1}^{\theta_k}
	 			\ell_{i_r}(\widetilde\sigma_{i_1\cdots i_{r-1}}\omega) 
	 		}    .
	 	\end{aligned}
	\end{equation} 

	Let $T:{\mathcal E}\to \mathcal E$ be the shift map defined in \eqref{def-t}. Define
	$X :{\mathcal E }\to\mathbb R$ by
	$$
	X(\omega,\mathbf d):=\log\ell_{\tau(d_1)}(\omega).
	$$
	Then {$X(T^{r-1}(\omega,\mathbf d))=X((\widetilde{\sigma}_{\tau(d_1\ldots d_{r-1})}\omega, \sigma^{r-1}\mathbf d))=\log\ell_{\tau(d_r)}(\widetilde{\sigma}_{\tau(d_1\ldots d_{r-1})}\omega)$}. Clearly, $X$ is bounded. So, for every $\omega\in\Omega_1$ and every $\mathbf d\in\Sigma^\omega$,
	taking logarithms in \eqref{eq-muQ1} and dividing by $-k\log m$ gives
	\begin{equation}	\label{eq-local-expression-level}
		\begin{aligned}
			\liminf_{k\to\infty}
			\frac{\log\mu_\omega\Pi^{-1}(Q_k(\mathbf d))}{-k\log m}
			&= \frac{\log \overline L}{\log m}+
			\liminf_{k\to\infty}
			\frac{1}{ \log m}
			\bigg(   \frac{\log M(\omega)}{k}  -   \frac{ \log    M(\widetilde\sigma_{i_1\cdots i_{k}}\omega)}{ k}      \\
			&\qquad+
			\frac{1}{k}
			\sum_{r=1}^{\theta_k}X(T^{r-1}(\omega,\mathbf d))	- \frac{\lambda}{k}
			\sum_{r=1}^{k}X(T^{r-1}(\omega,\mathbf d))
			\bigg)  .
		\end{aligned}
	\end{equation}
	
	We   first prove the upper estimate. {Since $\lim_{k\to\infty}\log M(\omega)/k=0$ for any $\omega\in\Omega_1$, by
	\eqref{eq-uniform} and \eqref{eq-local-expression-level} it follows that for any $\omega\in\Omega_1$ and any $\mathbf d\in\Sigma^\omega$,}
 \begin{equation*}	\label{eq-upper-local-level} 
	\begin{aligned}
		&	\liminf_{k\to\infty}
		\frac{\log\mu_\omega\Pi^{-1}(Q_k(\mathbf d))}{-k\log m}\\
		&\qquad \leq \frac{\log \overline L}{\log m}  -	\limsup_{k\to\infty}
	\frac{1}{k }
		\bigg(\frac{\lambda}{\log m}
		\sum_{r=1}^{k}X(T^{r-1}(\omega,\mathbf d))-
		\frac{1}{\log m}
		\sum_{r=1}^{\lfloor\lambda k\rfloor}X(T^{r-1}(\omega,\mathbf d)) \bigg)\\ 
		& \qquad \le
		\frac{\log \overline L}{\log m},
	\end{aligned}
\end{equation*} where the last inequality follows by Lemma \ref{lem-limsup}  and the boundedness of $X$. 
Since
$\mathbf d\in\Sigma^\omega$ was arbitrary, it follows from   Lemma~\ref{lem-bil}  that
	\begin{equation}\label{eq-upper-final}
		\dim_H K_\omega(\A)
		\le
		\frac{\log \overline L}{\log m}\qquad\forall~\omega\in\Omega_1.
	\end{equation}

	We next prove the lower estimate.  By Proposition~\ref{pro-ergodic}, the measure $\widetilde\mu$ is
	$T$-ergodic.  Applying Birkhoff's ergodic theorem   to
	$(\mathcal E,T,\widetilde\mu)$, it follows that for $\widetilde\mu$-a.e.\
	$(\omega,\mathbf d)\in\mathcal E_0$,
	$$
	\lim_{k\to\infty}
	\frac{1}{k}\sum_{r=1}^{k}X(T^{r-1}(\omega,\mathbf d))
	=
	\int_{\mathcal E}X(\omega,\mathbf d)\,d\widetilde\mu(\omega,\mathbf d).
	$$
Since $\theta_k{=\lfloor\lambda k\rfloor}\to\infty$ and $\theta_k/k\to\lambda$, the two
averages in \eqref{eq-local-expression-level} cancel in the limit.
Together with Lemma~\ref{lem-M0}, this gives {that for $\widetilde{\mu}$-a.e.~$(\omega,\mathbf d)\in\mathcal E_0$,}
	\begin{equation}\label{eq-lower-local}\begin{aligned}
			\liminf_{k\to\infty}
			\frac{\log\mu_\omega\Pi^{-1}(Q_k(\mathbf d))}{-k\log m}
			&=
			\frac{1}{\log m}
			\left(\log \overline L
			- \lambda\int_{\mathcal E}X \, d\widetilde\mu
			+ \lambda\int_{\mathcal E} X\, d\widetilde\mu
			\right) \\
			&  =
			\frac{\log \overline L}{\log m}.
		\end{aligned} 
	\end{equation}
	Since
	$d\widetilde\mu(\omega,\mathbf d)=M(\omega)\,d\mu_\omega(\mathbf d)\,d\mathbb P(\omega)$
	and $M>0$ almost surely, Fubini's theorem implies that
	\eqref{eq-lower-local} holds for $\mathbb P$-a.e.\,$\omega$ and
	$\mu_\omega$-a.e.\,$\mathbf d\in\Sigma^{\omega}$. The lower-bound part of
	Lemma~\ref{lem-bil} therefore gives
	\begin{equation}\label{eq-lower-final}
		\dim_H K_\omega(\mathcal A)
		\ge
		\frac{\log \overline L}{\log m}
		\qquad
		\text{for }\mathbb P \text{-a.e.  }\omega.
	\end{equation}
	Combining \eqref{eq-upper-final} and \eqref{eq-lower-final}, and using
	\eqref{eq-c0}, proves the theorem. 
\end{proof}
 
\medskip
{\noindent\bf Acknowledgements.}
 Kong was supported by Chongqing NSF: CQYC20220511052 and Scientific Research Innovation Capacity Support Project for Young Faculty No.~ZYGXQNISKYCXNLZCXM-P2P. Ma was supported by NKFIH grants K142169 and KKP144059. Simon was supported by NKFIH grants KKP144059 and Advanced grant 152858.

\medskip
{\noindent\bf Declaration of AI use.} 
 The {authors} used ChatGPT (OpenAI) to improve the
language and organization of the manuscript and to assist in
identifying relevant references. ChatGPT also assisted in developing
the proof of Lemma~\ref{lem-uniform} and in developing a rigorous
proof of Lemma~\ref{lem-G} based on the authors' intuition.
All mathematical arguments and cited references were independently
checked by the {authors}, who also reviewed and edited all AI-assisted
material and {took} full responsibility for the content of the paper.

 \FloatBarrier


\begin{thebibliography}{99}
	
	\bibitem{Baranski2007}
	K.~Bara\'nski,
	\emph{Hausdorff dimension of the limit sets of some planar geometric constructions},
	Adv. Math. \textbf{210} (2007), 215--245.
	
	\bibitem{Bedford1984}
	T.~J.~Bedford,
	\emph{Crinkly curves, Markov partitions and dimension},
	Ph.D. thesis, University of Warwick, 1984.
	
	\bibitem{Billingsley1965}
	P.~Billingsley,
	\emph{Ergodic Theory and Information},
	John Wiley \& Sons, New York, 1965.
	
	\bibitem{Bruin2018}
	H.~Bruin, M.~F.~Demers and M.~Todd,
	\emph{Hitting and escaping statistics: mixing, targets and holes},
	Adv. Math. \textbf{328} (2018), 1263--1298.
	
	\bibitem{Demers2005}
	M.~F.~Demers,
	\emph{Markov extensions for dynamical systems with holes: an application to expanding maps of the interval},
	Israel J. Math. \textbf{146} (2005), 189--221.
	
	\bibitem{DemersYoung2006}
	M.~F.~Demers and L.-S.~Young,
	\emph{Escape rates and conditionally invariant measures},
	Nonlinearity \textbf{19} (2006), 377--397.
	
	\bibitem{Falconer1986}
	K.~J.~Falconer,
	\emph{Random fractals},
	Math. Proc. Cambridge Philos. Soc. \textbf{100} (1986), 559--582.
	
	\bibitem{Falconer2003}
	K.~J.~Falconer,
	\emph{Fractal Geometry: Mathematical Foundations and Applications},
	2nd ed., John Wiley \& Sons, Chichester, 2003.
	
	\bibitem{FalconerJin2014}
	K.~J.~Falconer and X.~Jin,
	\emph{Exact dimensionality and projections of random self-similar measures and sets},
	J. Lond. Math. Soc. (2) \textbf{90} (2014), 388--412.
	
	\bibitem{FergusonJordanRams2015}
	A.~Ferguson, T.~Jordan and M.~Rams,
	\emph{Dimension of self-affine sets with holes},
	Ann. Acad. Sci. Fenn. Math. \textbf{40} (2015), 63--88.
	
	\bibitem{Ferguson2012}
	A.~Ferguson and M.~Pollicott,
	\emph{Escape rates for Gibbs measures},
	Ergodic Theory Dynam. Systems \textbf{32} (2012), 961--988.
	
	\bibitem{Graf1987}
	S.~Graf,
	\emph{Statistically self-similar fractals},
	Probab. Theory Related Fields \textbf{74} (1987), 357--392.
	
	\bibitem{Graf1988}
	S.~Graf, R.~D.~Mauldin and S.~C.~Williams,
	\emph{The exact Hausdorff dimension in random recursive constructions},
	Mem. Amer. Math. Soc. \textbf{71} (1988), no.~381, x+121~pp.
	
	\bibitem{Gui2008}
	Y.-X.~Gui and W.-X.~Li,
	\emph{A random version of McMullen--Bedford general Sierpinski carpets and its application},
	Nonlinearity \textbf{21} (2008), 1745--1758.
	
	\bibitem{Kenyon1996}
	R.~Kenyon and Y.~Peres,
	\emph{Measures of full dimension on affine-invariant sets},
	Ergodic Theory Dynam. Systems \textbf{16} (1996), 307--323.
	
	\bibitem{Kong2026}
	D.-R.~Kong, B.-B.~Sun and Z.-Q.~Wang,
	\emph{Open dynamical systems with a moving hole},
	Adv. Math. \textbf{484} (2026), Article~110710.
	
	\bibitem{GatzourasLalley1992}
	S.~P.~Lalley and D.~Gatzouras,
	\emph{Hausdorff and box dimensions of certain self-affine fractals},
	Indiana Univ. Math. J. \textbf{41} (1992), 533--568.
	
	\bibitem{MauldinWilliams1986}
	R.~D.~Mauldin and S.~C.~Williams,
	\emph{Random recursive constructions: asymptotic geometric and topological properties},
	Trans. Amer. Math. Soc. \textbf{295} (1986), 325--346.
	
	\bibitem{McMullen1984}
	C.~McMullen,
	\emph{The Hausdorff dimension of general Sierpi\'nski carpets},
	Nagoya Math. J. \textbf{96} (1984), 1--9.
	
	\bibitem{Pianigiani1979}
	G.~Pianigiani and J.~A.~Yorke,
	\emph{Expanding maps on sets which are almost invariant. Decay and chaos},
	Trans. Amer. Math. Soc. \textbf{252} (1979), 351--366.
	
\end{thebibliography}
\end{document}